\documentclass[reqno,a4paper,11pt]{amsart}

\usepackage[all,poly]{xy}
\usepackage{amsfonts}
\usepackage[mathcal]{eucal}
\usepackage{eufrak}
\usepackage{mathrsfs}
\usepackage{color}
\usepackage{hyperref}
\usepackage{enumerate}

\usepackage{diagrams}

\DeclareFontFamily{OT1}{pzc}{}
\DeclareFontShape{OT1}{pzc}{m}{it}{<-> s * [1.10] pzcmi7t}{}
\DeclareMathAlphabet{\mathpzc}{OT1}{pzc}{m}{it}

\theoremstyle{plain}
\newtheorem {lemma}{Lemma}[] 
\newtheorem {theorem}[lemma]{Theorem}

\newtheorem {corollary}[lemma]{Corollary}

\newtheorem {proposition}[lemma]{Proposition}

\theoremstyle{definition}

\theoremstyle{definition}

{}

\newcommand{\Ga}{\Gamma}
\newcommand{\ga}{\gamma}

\newcommand{\Pgr}{\mathpzc{Pgr}}
\newcommand{\Prr}{\mathpzc{Pr}}
\newcommand{\gr}{\text{gr}}
\newcommand{\coker}{\text{coker}\,}

\let\nbd=\nobreakdash
\newcommand{\ie}{{\it i.e.}}

\title{On Quillen's calculation of graded $K$-theory}

\author{Roozbeh Hazrat}

\address{Roozbeh Hazrat, 
Pure Mathematics Research Centre,
School of Mathematics and Physics,
Queen's University Belfast, Belfast
  BT7~1NN, Northern Ireland, United Kingdom}

\email{r.hazrat@qub.ac.uk}

\urladdr{http://web.am.qub.ac.uk/users/r.hazrat/Home.html}

\author{Thomas H\"uttemann}

\address{Thomas H\"uttemann, Pure Mathematics Research Centre,
School of Mathematics and Physics,
Queen's University Belfast, Belfast
  BT7~1NN, Northern Ireland, United Kingdom}

\email{t.huettemann@qub.ac.uk}

\urladdr{http://huettemann.zzl.org/}



\thanks{This work was supported by the Engineering and Physical
  Sciences Research Council [grants number EP/H018743/1 and
  EP/I007784/1].}

\begin{document}

\begin{abstract}
  We adapt \textsc{Quillen}'s calculation of graded $K$\nbd-groups of
  $\mathbb Z$\nbd-graded rings with support in~$\mathbb N$ to graded
  $K$\nbd-theory, allowing gradings in a product $\mathbb Z \times G$
  with $G$ an arbitrary group. This in turn allows us to use inductions and calculate graded $K$-theory of $\mathbb Z^m$-graded rings. 
\end{abstract}

\maketitle

\section*{Introduction}

Let $G$ be a group, written additively, and let $A$ be a
$G$\nbd-graded ring. We denote the category of finitely generated
$G$\nbd-graded projective right $A$\nbd-modules by~$\Pgr^G(A)$. This
is an exact category with the usual notion of (split) short exact
sequence, and we denote the \textsc{Quillen} $K$\nbd-groups
$K_i(\Pgr^G(A))$ by~$K_i^G(A)$. The group $G$ acts on the category
$\Pgr^G(A)$ from the right via $(P,g) \mapsto P(g)$, where $P(g)_h =
P_{g+h}$. By functoriality of $K$\nbd-groups this equips $K_i^G(A)$
with the structure of a right $\mathbb Z[G]$\nbd-module.

If $A$ is strongly graded, \ie, if $1 \in A_g A_{-g}$ for all $g \in
G$, then \textsc{Dade}'s Theorem (\cite[Theorem~3.11]{grrings})
implies that the category of finitely generated projective right
$A_0$-modules, denoted $\Prr(A_0)$, is equivalent to~$\Pgr^G(A)$. This
implies that there is a natural isomorphism $K_i(A_0) \cong K_i^G(A)$.

The relation between graded $K$\nbd-groups and non-graded
$K$\nbd-groups is not always apparent. For example consider the
$\mathbb Z$\nbd-graded matrix ring $A=\mathbb{M}_5(F)(0,1,2,2,3)$, 
where $F$ is a field (see \cite[\S 9.2]{grrings} for details). Using
graded \textsc{Morita} theory one can show that
$K_0(A_0)\cong \mathbb Z \times \mathbb Z \times \mathbb Z\times
\mathbb Z$, whereas $K_0^{\mathbb Z}(A)\cong \mathbb Z[x,x^{-1}]$. Note
also that $K_0(A) \cong \mathbb{Z}$.

For a $\mathbb Z$\nbd-graded ring~$A$ with support in~$\mathbb N$ the
graded $K$\nbd-theory of~$A$ was determined by \textsc{Quillen}
\cite[Proposition, p.~107]{quillen}: The functor $P \mapsto P
\otimes_{A_0} A$ induces an isomorphism of $\mathbb
Z[x,x^{-1}]$\nbd-modules
\begin{equation}\label{qgrgr} K_i(A_0) \otimes_{\mathbb Z} \mathbb
  Z[x,x^{-1}] \cong K_i^{\mathbb Z}(A)  \ .
\end{equation}

Contrary to other fundamental theorems in the subject, such as fundamental theorem of $K$-theory (i.e., $K_i(R[x,x^{-1}])=K_i(R)\times K_{i-1}(R)$, for $R$ a regular ring), one can not use an easy induction on  ~(\ref{qgrgr}) to write a similar statement for ``multi-variables'' rings. For example, it appears that there is no obvious inductive approach to
generalise~(\ref{qgrgr}) to $\mathbb Z^m \times G$\nbd-graded
rings. However, by generalising \textsc{Quillen}'s argument to take
gradings into account on both sides of the isomorphism, such a
procedure becomes feasable. We will prove the following statement:

\begin{theorem}
  \label{thm:main}
  Let $G$ be a group, and let $A$ be a ${\mathbb Z} \times
  G$\nbd-graded ring with support in ${\mathbb N} \times G$. Then
  there is a $\mathbb Z[{\mathbb Z}\times G]$-module isomorphism
  \[K_i^{G}(A_{(0,-)}) \otimes_{\mathbb Z[G]} \mathbb Z[{\mathbb Z}\times G]
  \cong K_i^{{\mathbb Z} \times G}(A) ,\] where $A_{(0,-)}=\bigoplus_{g\in G}
  A_{(0,g)}$.
\end{theorem}

By a straightforward induction this now implies:

\begin{corollary}
  For a $\mathbb Z^m \times G$\nbd-graded ring~$A$ with support in
  $\mathbb N^m \times G$ there is a $\mathbb Z[x_1,\, x_1^{-1},\,
  x_2,\, x_2^{-1},\, \cdots,\, x_m,\, x_m^{-1}]$-module isomorphism
  \[K_i^G(A_{(0,-)}) \otimes_{\mathbb Z} \mathbb Z[x_1,\, x_1^{-1},\,
  x_2,\, x_2^{-1},\, \cdots,\, x_m,\, x_m^{-1}] \cong K_i^{\mathbb
    Z^m \times G}(A) \ .\]
\end{corollary}
For a trivial group~$G$ this is a direct generalisation of
\textsc{Quillen}'s theorem to $\mathbb Z^m$\nbd-graded rings.

\section*{Swan's Theorem}

As in \textsc{Quillen}'s calculation the proof of the Theorem is based
on a version of \textsc{Swan}'s Theorem, modified to the present
situation: it provides a correspondence between isomorphism classes of
$\mathbb Z \times G$\nbd-finitely generated graded projective
$A$\nbd-modules and of $G$\nbd-finitely generated graded projective
$A_{(0,-)}$\nbd-modules.

\begin{proposition}
  \label{swanii} Let $\Ga$ be a (possibly non-\textsc{abel}ian)
  group. Let $A$ be a $\Ga$\nbd-graded ring, $A_0$ a graded subring
  of~$A$ and $\pi \colon A\rightarrow A_0$ a graded ring homomorphism
  such that $\pi |_{A_0}=1$. (In other words, $A_0$~is a retract
  of~$A$ in the category of $\Ga$\nbd-graded rings.) We denote the
  kernel of~$\pi$ by~$A_+$.

  Suppose that for any finitely generated graded right $A$\nbd-module
  $M$ the condition $MA_+=M$ implies $M=0$. Then the natural
  functor
  \[S = - \otimes_{A_0} A \colon \Prr^{\Ga}(A_0) \rightarrow
  \Prr^{\Ga}(A)\] induces a bijective correspondence between the
  isomorphism classes of finitely generated graded projective
  $A_0$\nbd-modules and of finitely generated graded projective
  $A$\nbd-modules. An inverse of the bijection is given by the functor
  \[T = -\otimes_A A_0 \colon \Prr^{\Ga}(A) \rightarrow
  \Prr^{\Ga}(A_0) \ .\] There is a natural isomorphism $T \circ S
  \cong \mathrm{id}$, and for each $P \in \Prr^\Ga (A)$ a
  non-canonical isomorphism $S \circ T(P) \cong P$. The latter is
  given by
  \begin{equation}
    \label{eq:iso_lemma}
    T(P) \otimes_{A_0} A \rightarrow P \ , \quad x \otimes a \mapsto
    g(x) \cdot a,
  \end{equation}
  where $g$~is an $A_0$\nbd-linear section of the epimorphism $P
  \rightarrow T(P)$.
\end{proposition}

\begin{proof}
  For any finitely generated graded projective $A_0$\nbd-module~$Q$ we
  have a natural isomorphism $TS(Q) \cong Q$ given by
  \begin{equation}
    \label{eq:nuQ}
    \nu_Q \colon TS(Q) = Q \otimes_{A_0} A \otimes_A {A_0} \rightarrow
    Q \ , \quad q \otimes a \otimes a_0 \mapsto q \pi(a)a_0 \ .
  \end{equation}
  We will show that for a graded projective $A$\nbd-module~$P$
  there is a non-canonical graded isomorphism $ST(P)\cong_{\gr}
  P$. The lemma then follows.

  Consider the natural graded $A$\nbd-module epimorphism
  \[f \colon P\rightarrow T(P)=P\otimes_{A} A_0 \ , \quad p \mapsto p
  \otimes 1 \ .\] Here $T(P)$ is considered as an $A$\nbd-module via
  the map~$\pi$. Since $T(P)$ is a graded projective $A_0$\nbd-module,
  the map $f$~has a graded $A_0$\nbd-linear section $g \colon T(P)
  \rightarrow P$. This section determines an $A$\nbd-linear graded map
  \[\psi \colon ST(P) = P \otimes_A A_0 \otimes_{A_0} A \rightarrow P
  \ , \quad p \otimes a_0 \otimes a \mapsto g(p \otimes a_0) \cdot a \
  ,\] and we will show that $\psi$~is an isomorphism. --- First note
  that the map
  \[T(f) \colon T(P) \rightarrow TT(P) \ , \quad p \otimes a_0 \mapsto
  f(p) \otimes 1 = p \otimes 1 \otimes a_0\] is an isomorphism
  (consider $T(P)$~as an $A$\nbd-module via~$\pi$ here). In fact the
  inverse is given by the isomorphism $TT(P) = P \otimes_A A_0 \otimes_A
  A_0 \rTo P \otimes_A A_0$ which maps $p \otimes a_0 \otimes b_0$ to
  $p \otimes(a_0 b_0)$.  Tracing the definitions now shows that both
  composites
  \[TST(P) \pile{\rTo^{T(\psi)} \\ \rTo_{\nu_{T(P)}}} T(P)
  \rTo^{T(f)}_\cong TT(P)\] map $p \otimes a_0 \otimes a \otimes b_0
  \in P \otimes_A A_0 \otimes_{A_0} A \otimes_A A_0 = TST(P)$ to the
  element $f \big( g(p \otimes a_0) \cdot a \big) \otimes b_0 = p
  \otimes (a_0 \pi(a) b_0) \otimes 1 \in TT(P)$. This implies that
  $T(\psi) = \nu_{T(P)}$, which is an isomorphism.

  The exact sequence
  \begin{equation}
    \label{gfds} 0 \rightarrow \ker \psi \rightarrow ST(P)
    \overset{\psi}{\rightarrow} P \rightarrow \coker \psi \rightarrow
    0
  \end{equation}
  gives rise, upon application of the right exact functor~$T$, to an
  exact sequence
  \[TST(P) \overset{T(\psi)}{\rightarrow} T(P) \rightarrow T(\coker
  \psi) \rightarrow 0 \ .\] Since $T(\psi)$ is an isomorphism we have
  $T(\coker \psi) = \coker T(\psi) = 0$. Since $\coker \psi$ is
  finitely generated by~\eqref{gfds} this implies $\coker \psi = 0$
  (note that $T(M) = MA_+$ for every finitely generated
  module~$M$). In other words, $\psi$~is surjective and
  \eqref{gfds}~becomes the short exact sequence $0 \rightarrow \ker
  \psi \rightarrow ST(P) \overset{\psi}{\rightarrow} P \rightarrow
  0$. This latter sequence splits since $P$~is projective; this
  immediately implies that $\ker \psi$ is finitely generated, and
  since $T(\psi)$ is injective we also have $T(\ker \psi) = \ker
  T(\psi) = 0$. The hypotheses guarantee $\ker \psi = 0$ now so that
  $\psi$~is injective as well as surjective, and thus is an
  isomorphism as claimed.
\end{proof}


\section*{A lemma on graded $K$-theory}

\begin{lemma}
  \label{hggf}
  Let $G$ and~$\Ga$ be groups, and let $A$ be a $G$\nbd-graded
  ring. Then, considering $A$ as a $\Ga \times G$\nbd-graded ring in a
  trivial way where necessary, the functorial assignment $(M, \gamma)
  \mapsto M(\gamma,0)$ induces a $\mathbb Z[\Ga\times G]$\nbd-module
  isomorphism
  \[K_i^{\Ga \times G}(A) \cong K_i^{G}(A) \otimes_{\mathbb Z[G]}
  \mathbb Z[\Ga \times G] \ .\]
\end{lemma}

\begin{proof}
  Let $P = \bigoplus_{(\gamma,g) \in \Ga \times G} P_{(\gamma,g)}$ be
  a $\Gamma\times G$\nbd-finitely generated graded projective
  $A$\nbd-module. Since the support of $A$ is $G = 1 \times G$, there
  is a unique decomposition $P=\bigoplus_{\ga \in \Ga} P_\ga$, where
  the $P_\ga = \bigoplus_{g \in G} P_{(\gamma,g)}$ are finitely
  generated $G$\nbd-graded projective  $A$\nbd-modules. This gives a
  natural isomorphism of categories
  \[\Psi \colon \Pgr^{\Gamma \times G} (A)
  \overset{\cong}{\longrightarrow} \bigoplus_{\gamma \in \Gamma}
  \Pgr^{G}(A) \ .\]
  The natural right action of $\Ga \times G$ on these categories is
  described as follows: for a given module $P \in \Pgr^{\Ga \times G}
  (A)$ as above and elements $(\gamma,g) \in \Ga \times G$ we have
  \[P(\gamma,g)_{(\delta,h)} = P_{(\gamma+\delta,g+h)} \text{ and }
  \Psi(P)_\delta = P_{\gamma + \delta}(g)\] so that $\Psi\big(
  P(\gamma,g) \big) = \Psi(P)(\gamma,g)$. Since $K$-groups respect
  direct sums we thus have a chain of $\mathbb Z[\Ga \times
  G]$\nbd-linear isomorphisms
  \begin{multline*}
    K_i^{\Ga\times G}(A) = K_i(\Pgr^{\Gamma \times G} (A) )\cong
    K_i\big(\bigoplus_{\gamma \in \Gamma} \Pgr^{G}(A)\big) \\
    = \bigoplus_{\ga \in \Ga} K_i^G(A) \cong K_i^G(A) \otimes_{\mathbb
      Z} \mathbb Z[\Ga] = K_i^G(A) \otimes_{\mathbb
      Z[G]} \mathbb Z[\Ga \times G]\ .
  \end{multline*}
\end{proof}

\section*{Proof of Theorem~\ref{thm:main}}

Let $A$ be a ${\mathbb Z} \times G$\nbd-graded ring with
support in ${\mathbb N} \times G$. That is, $A$~comes equipped with a
decomposition
\[A=\bigoplus_{\omega\in{\mathbb N}} A_{(\omega,-)} \quad \text{where}
\quad A_{(\omega,-)} = \bigoplus_{g \in G} A_{(\omega,g)}\ .\] The
ring~$A$ has a ${\mathbb Z} \times G$\nbd-graded subring~$A_{(0,-)}$ (with
trivial grading in ${\mathbb Z}$\nbd-direction). The projection map $A
\rightarrow A_{(0,-)}$ is a $\mathbb Z\times G$-graded  ring homomorphism; its
kernel is denoted~$A_+$. Explicitly, $A_+$~is the two-sided ideal
\[A_+=\bigoplus_{\omega > 0} A_{(\omega,-)} \ .\] We
identify the quotient ring~$A/A_+$ with the subring~$A_{(0,-)}$ via
the projection.

\subsection*{Functors}

If $P$ is a finitely generated graded projective $A$\nbd-module, then
$P\otimes_A A_{(0,-)}$ is a finitely generated ${\mathbb Z}\times
G$\nbd-graded projective $A_{(0,-)}$\nbd-module. Similarly, if $Q$~is a
finitely generated graded projective $A_{(0,-)}$\nbd-module then $Q
\otimes_{A_{(0,-)}} A$ is a $\mathbb Z\times G$- finitely generated graded projective
$A$\nbd-module. We can thus define functors
\begin{align*}
  T & =-\otimes_A A_{(0,-)} \colon \Pgr^{{\mathbb Z} \times G}(A)
  \longrightarrow \Pgr^{{\mathbb Z} \times G}(A_{(0,-)}) \\
  \text{and} \quad S &= -\otimes_{A_{(0,-)}}A \colon \Pgr^{{\mathbb Z} \times
    G}(A_{(0,-)}) \longrightarrow \Pgr^{{\mathbb Z} \times G}(A)\ .
\end{align*}
Since $T(P) = P/PA_+$ we see that {\it the support of~$T(P)$ is
  contained in the support of~$P$}.

Observe now that {\it if $M$ is a finitely generated ${\mathbb Z}\times
  G$-graded $A$-module and $MA_+=M$ then $M=0$}; for if $M \neq 0$
there is a minimal $\omega \in {\mathbb Z}$ such that $M_{(\omega,-)} \neq
0$, but $(MA_+)_{(\omega,-)} = 0$.  It follows from
Proposition~\ref{swanii} that for each graded finitely generated
projective $A$\nbd-module~$P$ there is a non-canonical isomorphism $P
\cong T(P) \otimes_{A_{(0,-)}} A$ as in \eqref{eq:iso_lemma} which
respects the ${\mathbb Z} \times G$\nbd-grading. Explicitly, for a given
$(\omega,g) \in {\mathbb Z} \times G$ we have an isomorphism of
\textsc{abel}ian groups
\begin{equation}
  \label{eq:noncan_iso}
  P_{(\omega,g)} \cong \bigoplus_{(\kappa,h)} T(P)_{(\kappa,h)}
  \otimes A_{(-\kappa+\omega,-h+g)} \ ;
\end{equation}
the tensor product $T(P)_{(\kappa,h)} \otimes
A_{(-\kappa+\omega,-h+g)}$ denotes, by convention, the
\textsc{abel}ian subgroup of $T(P) \otimes_{A_{(0,-)}} A$ generated by
primitive tensors of the form $x \otimes y$ with homogeneous elements
$x \in T(P)$ of degree~$(\kappa,h)$ and $y \in A$ of
degree~$(-\kappa+\omega,-h+g)$.

\subsection*{Filtration}

For a ${\mathbb Z} \times G$\nbd-graded $A$\nbd-module~$P$ write $P =
\bigoplus_{\omega \in {\mathbb Z}} P_{(\omega,-)}$, where $P_{(\omega,-)} =
\bigoplus_{g \in G} P_{(\omega,g)}$. For $\lambda \in {\mathbb Z}$ let
$F^\lambda P$ denote the $A$\nbd-submodule of~$P$ generated by the
elements of $\bigcup_{\omega \leq \lambda} P_{(\omega,-)}$; this is
${\mathbb Z} \times G$\nbd-graded again. As an explicit example, we have
\[F^\lambda A(\omega,g) = \begin{cases} A(\omega,g) & \text{if }
  \lambda \geq -\omega \ , \\ 0 & \text{else.} \end{cases}\] 

Suppose that $P$~is a finitely generated graded projective
$A$-module. Since the support of~$A$ is contained in ${\mathbb N}
\times G$ there exists $n \in {\mathbb Z}$ such that $F^{-n}P = 0$ and
$F^{n}P = P$. Write $\Pgr_n^{{\mathbb Z} \times G}(A)$ for the full
subcategory of~$\Pgr^{{\mathbb Z} \times G}(A)$ spanned by those
modules~$P$ which satisfy $F^{-n}P = 0$ and $F^{n}P = P$. Then
$\Pgr^{{\mathbb Z} \times G}(A)$ is the filtered union of the
$\Pgr_n^{{\mathbb Z} \times G}(A)$.

Let $P \in \Pgr_n^{{\mathbb Z} \times G}(A)$; we want to identify
$F^\lambda P$. By definition, the $A$\nbd-module $F^\lambda P$ is
generated by the elements of~$P_{(\omega, g)}$ for $\omega \leq
\lambda$, with $P_{(\omega,g)}$ having been identified
in~\eqref{eq:noncan_iso}. We remark that the direct summands
in~\eqref{eq:noncan_iso} indexed by $\kappa > \omega$ are trivial as
$A$~has support in ${\mathbb N} \times G$. On the other hand, for
$\omega \geq \kappa$ a given primitive tensor $x \otimes y \in
P_{(\omega,g)}$ with $x \in T(P)_{(\kappa,h)}$ and $y \in
A_{(-\kappa+\omega, -h+g)}$ can always be re-written, using the right
$A$\nbd-module structure of $T(P) \otimes_{A_{(0,-)}} A$, as
\[x \otimes y = (x \otimes 1) \cdot y \qquad \text{where } x \otimes 1
\in T(P)_{(\kappa,h)} \otimes A_{(0,0)} \subseteq P_{(\kappa, h)} \
.\] That is, the $A$\nbd-module $F^\lambda P$ is generated by those
summands of \eqref{eq:noncan_iso} with $\kappa = \omega \leq
\lambda$. We claim now that $F^\lambda P$ is isomorphic to
\begin{equation}
  \label{eq:M}
  M^{(\lambda)} = \bigoplus_{\kappa \leq \lambda} T(P)_{(\kappa,-)}
  \otimes_{A_{(0,-)}} A(-\kappa, 0) \ ,
\end{equation}
considering $T(P)_{(\kappa,-)}$ as a ${\mathbb Z} \times G$\nbd-graded
$A_{(0,-)}$\nbd-module with support in $\{0\}\times G$.  The
homogeneous components of~$M^{(\lambda)}$ are given by
\[M^{(\lambda)}_{(\omega,g)} = \bigoplus_{\kappa \leq \lambda}
\bigoplus_{h \in G} T(P)_{(\kappa,h)} \otimes
A(-\kappa,0)_{(\omega,-h+g)} \ .\] Now elements of the form
\[x \otimes 1 \in T(P)_{(\kappa,h)} \otimes
A(-\kappa,0)_{(\kappa,-h+h)} \subseteq M^{(\lambda)}_{(\kappa,h)}\]
clearly form a set of $A$\nbd-module generators for~$M^{(\lambda)}$ so
that, by the argument given above, $F^\lambda P$ and~$M^{(\lambda)}$
have the same generators in the same degrees. The claim follows. ---
The module $F^\lambda P$ is finitely generated ({\it viz.}, by those
generators of~$P$ that have ${\mathbb Z}$\nbd-degree at
most~$\lambda$). Since $T(P)$~is a finitely generated projective
$A_{(0,-)}$\nbd-module so is its summand $T(P)_{(\lambda_k,-)}$;
consequently, $P \mapsto F^kP$~is an endofunctor of~$\Pgr_n^{{\mathbb Z}
  \times G} (A)$. It is exact as can be deduced from the
(non-canonical) isomorphism in~\eqref{eq:M}, using exactness of tensor
products.

\subsection*{Filtration quotients}

From the isomorphism $F^k P \cong M^{(\lambda_k)}$, cf.~\eqref{eq:M},
we obtain an isomorphism
\begin{equation}
  \label{eq:Q}
  F^{k+1}P / F^kP \cong T(P)_{(\lambda_k,-)} \otimes_{A_{(0,-)}}
  A(-\lambda_k, 0) \ ;
\end{equation}
in particular, $F^{k+1}P/F^kP \in \Pgr_n^{{\mathbb Z} \times G} (A)$.

The isomorphism~\eqref{eq:Q} depends on the isomorphism~\eqref{eq:M},
and thus ultimately on~\eqref{eq:iso_lemma}. The latter depends on a
choice of a section~$g$ of $P \rightarrow T(P)$. Given another
section~$g_0$ the difference $g-g_0$ has image in $\ker (P \rightarrow
T(P)) = PA_+$. Since $A_+$~consists of elements of positive $\mathbb
Z$\nbd-degree only, this implies that the isomorphism $F^{k+1} P \cong
M^{(\lambda_{k+1})}$ does not depend on~$g$ up to elements in~$F^kP$;
in other words, the quotient $F^{k+1}P/F^kP$ is independent of the
choice of~$g$. Thus the isomorphism~\eqref{eq:Q} is, in fact,
a natural isomorphism of functors.

\subsection*{$K$-theory}

We are now in a position to perform the $K$\nbd-theoretical
calculations. First define the exact functor
\begin{align*} \Theta_q: \Pgr_q^{{\mathbb Z} \times G}(A_{(0,-)}) &
\longrightarrow \Pgr_q^{{\mathbb Z} \times G}(A)\\ P = \bigoplus_\omega
P_{(\omega,-)} & \mapsto \bigoplus_\omega P_{(\omega,-)}
\otimes_{A_{(0,-)}} A(-\omega,0) \ ;
\end{align*}
here $\Pgr_q^{{\mathbb Z} \times G}(A_{(0,-)})$ denotes the full
subcategory of~$\Pgr^{{\mathbb Z} \times G}(A_{(0,-)})$ spanned by
modules with support in $[-q,q] \times G$, and $P_{(\omega,-)}$~on the
right is considered as a ${\mathbb Z} \times G$\nbd-graded
$A_{(0,-)}$\nbd-module with support in~$\{0\} \times G$.

Next define the exact functor
\begin{align*} \Psi_q: \Pgr_q^{{\mathbb Z} \times G}(A) & \longrightarrow
\Pgr_q^{{\mathbb Z} \times G}(A_{(0,-)}) \\ P & \mapsto \bigoplus _{\omega
  \in {\mathbb Z}} T(P)_{(\omega,-)} \ ;
\end{align*}
here $T(P)_{(\omega,-)}$~is considered as an $A_{(0,-)}$\nbd-module
with support in~$\{\omega\} \times G$.

Now $\Psi_q \circ \Theta_q \cong \mathrm{id}$; indeed, the composition
sends the summand~$P_{(\omega,-)}$ of~$P$ to the $\kappa$\nbd-indexed
direct sum of
\[T\big(P_{(\omega,-)} \otimes_{A_{(0,-)}} A(-\omega,0)\big)_{(\kappa,-)}
\cong
\begin{cases}
  P_{(\omega,-)} & \text{if } \kappa = \omega \ , \\ 0 & \text{else.}
\end{cases}
\] In particular, $\Psi_q \circ \Theta_q$ induces the identity on
$K$\nbd-groups. --- As for the other composition, we have
\[\Theta_q \circ \Psi_q (P) = \bigoplus_{\omega} T(P)_{(\omega,-)}
\otimes_{A_{(0,-)}} A(-\omega,0) \underset{\eqref{eq:Q}}=
\bigoplus_{j=-q}^{q-1} F^{j+1}P / F^jP \ .\] Since $F^q =
\mathrm{id}$, additivity for characteristic filtrations \cite[p.~107,
Corollary~2]{quillen} implies that $\Theta_q \circ \Psi_q$ induces the
identity on $K$\nbd-groups.

For any $P \in \Pgr^{{\mathbb Z} \times G}(A_{(0,-)})$ we have
\[\big (P
\otimes_{A_{(0,-)}} A(-\omega,0)\big )(0,g)=P(0,g)\otimes _{A_{(0,-)}}
A(-\omega,0) \ ,\] by direct calculation. Hence the functor $\Theta_q$
induces a $\mathbb Z[G]$\nbd-linear isomorphism on
$K$\nbd-groups. Since $K$\nbd-groups are compatible with direct
limits, letting $q \rightarrow \infty$ yields a $\mathbb
Z[G]$\nbd-linear isomorphism $K_i^{{\mathbb Z}\times G}(A_{(0,-)}) \cong
K_i^{{\mathbb Z} \times G} (A)$
and thus, by Lemma~\ref{hggf}, a $\mathbb Z[{\mathbb Z} \times
G]$\nbd-module isomorphism
\[K_i^{G}(A_{(0,-)}) \otimes_{\mathbb Z[G]} \mathbb Z[{\mathbb Z}\times G]
\cong K_i^{{\mathbb Z} \times G}(A) \ .\]

\end{document}